\documentclass[11pt]{article}
\usepackage[utf8]{inputenc}
\usepackage[english]{babel}
\usepackage[none]{hyphenat}

\usepackage{amsmath}
\usepackage{amssymb}
\usepackage{amsthm}

\usepackage{amsfonts}

\usepackage{setspace}
\usepackage{fullpage}
\usepackage{graphicx}
\usepackage{url}
\usepackage{indentfirst}

\usepackage{xcolor}
\usepackage{graphicx}

\usepackage{float}

\usepackage{textcomp}
\usepackage{gensymb}
\usepackage{comment}

\usepackage{subcaption}
\usepackage[colorlinks=true,linkcolor=black, anchorcolor=black, citecolor=black, filecolor=black, menucolor=black, runcolor=black, urlcolor=black]{hyperref}
\usepackage{enumerate}

\usepackage{soul}
\usepackage{xcolor}

\usepackage{tikz-cd}

\usepackage{ dsfont }

\newtheorem{theorem}{Theorem}[section]
\newtheorem{corollary}[theorem]{Corollary}
\newtheorem{lemma}[theorem]{Lemma}

\theoremstyle{definition}
\newtheorem{definition}[theorem]{Definition}
\newtheorem{example}[theorem]{Example}

\newtheorem{rmk}[theorem]{Remark}

\theoremstyle{definition}

\usepackage[capitalize]{cleveref}
	\crefname{app-corollary}{Corollary}{Corollaries}
	\Crefname{app-corollary}{Corollary}{Corollaries}
	\crefname{app-definition}{Definition}{Definitions}
	\Crefname{app-definition}{Definition}{Definitions}
	\crefname{figure}{Figure}{Figures}
	\Crefname{figure}{Figure}{Figures}
	\crefname{lemma}{Lemma}{Lemmata}
	\Crefname{lemma}{Lemma}{Lemmata}
	\crefname{app-lemma}{Lemma}{Lemmata}
	\Crefname{app-lemma}{Lemma}{Lemmata}
	\crefname{app-proposition}{Proposition}{Proposition}
	\Crefname{app-proposition}{Proposition}{Proposition}
	\crefname{app-theorem}{Theorem}{Theorems}
	\Crefname{app-theorem}{Theorem}{Theorems}

\newcommand\Q{\mathbb{Q}}

\newcommand{\set}[1]{\left\{ #1 \right\} }

\newcommand{\given}{\;|\;}

\renewcommand{\epsilon}{\varepsilon}

\definecolor{highlight}{HTML}{cdeef2}

\newcommand{\Ass}{\mathrm{Ass}}
\newcommand{\Spec}{\mathrm{Spec}}
\newcommand{\Min}{\mathrm{Min}}
\renewcommand{\dim}{\mathrm{dim}}

\title{Completions of Quasi-excellent Domains}
\author{David Baron, Ammar Eltigani, S. Loepp, AnaMaria Perez, M. Teplitskiy}
\date{\today}

\begin{document}

\setlength{\parindent}{10 pt}
\setlength{\abovedisplayskip}{2pt}
\setlength{\belowdisplayskip}{2pt}
\setlength{\abovedisplayshortskip}{0pt}
\setlength{\belowdisplayshortskip}{0pt}

\setlength\fboxsep{10pt}

\maketitle


\begin{abstract}
    Let $T$ be a complete local (Noetherian) ring of characteristic zero.  We find necessary and sufficient conditions for $T$ to be the completion of a quasi-excellent local domain.  In the case that $T$ contains the rationals, we provide necessary and sufficient conditions for $T$ to be the completion of a \emph{countable} quasi-excellent local domain.  We also prove results regarding the possible lengths of maximal saturated chains of prime ideals of these quasi-excellent local domains, and we show that these results lead to interesting examples of noncatenary quasi-excellent local domains.
\end{abstract}

\section{Introduction}
In this paper, we aim to understand completions of quasi-excellent local domains. Completions of local (Noetherian) rings are an important tool in commutative algebra because complete local rings are very well understood via Cohen's structure theorem. If we want to understand certain properties of a local ring $A$, a common technique is to study the completion of the ring, $\widehat{A}$, with respect to its maximal ideal, in the hope that understanding $\widehat{A}$ will provide desired information about $A$. This process of passing to the completion, however, does not always act the way one might expect. For example, it is not difficult to find an example of a local domain whose completion is not a domain. In fact, in 1986, Lech proved the following remarkable theorem, showing that almost all complete local rings are the completion of a local domain. 

\begin{theorem}[\cite{Lech}, Theorem 1]\label{Lech's Theorem} 
    Let $T$ be a complete local ring with maximal ideal $M$. Then $T$ is the completion of a local domain if and only if
    \begin{enumerate}[(i)]
        \item No integer of $T$ is a zero divisor, and
        \item Unless equal to $(0)$, $M$ is not in $\Ass(T)$.
    \end{enumerate}
\end{theorem}

In \cite{Susan}, this result is extended to excellent local domains when the characteristic of $T$ is zero. In particular, in Theorem 9 from \cite{Susan}, it is shown that a complete local ring $T$ of characteristic zero is the completion of a excellent local domain if and only if

\begin{enumerate}[(i)]
    \item $T$ is reduced,
    \item $T$ is equidimensional, and
    \item no integer of $T$ is a zero divisor.
\end{enumerate}

One of the goals of this paper is to prove the analogous result for quasi-excellent local domains. That is, given a complete local ring $T$ of characteristic zero, we wish to find necessary and sufficient conditions for $T$ to be the completion of a quasi-excellent local domain. In fact, in Theorem \ref{quasi-excellent characterization}, we show that conditions (i) $T$ is reduced and (iii) no integer of $T$ is a zero divisor from Theorem 9 in \cite{Susan} are both necessary and sufficient. Recall from Theorem 3.4 in \cite{Michaelson} that if the completion of a local ring $R$ is equidimensional, then $R$ is universally catenary. Since quasi-excellent rings need not be universally catenary, it is not surprising that, for our theorem, the condition that $T$ is equidimensional is not included. 

The proof that conditions (i) and (iii) are necessary follows fairly quickly from results in the literature. However, showing that they are sufficient is more challenging. We consider the case where $\dim(T)=0$ and the case where $\dim(T) \geq 1$ separately. If $\dim(T)=0$ and $T$ is reduced, then $T$ is a field, so it is the completion of an excellent local domain (namely, itself). In the case where $\dim(T) \geq 1$, we use results from \cite{Susan} to construct a quasi-excellent local domain whose completion is $T$. 

It is worth noting that the technique used in \cite{Susan} to construct an excellent local domain $A$ whose completion is $T$ produces an $A$ that is uncountable. Therefore, it is interesting to ask what the necessary and sufficient conditions are on a complete local ring $T$ of characteristic zero for it to be the completion of a \textit{countable} quasi-excellent local domain. Suppose $T$ is a complete local ring with maximal ideal $M$ and assume that $T$ contains the rationals. It is shown in \cite{Teresa} that $T$ is the completion of a countable excellent local domain if and only if
\begin{enumerate}[(i)]
    \item $T$ is equidimensional,
    \item $T$ is reduced, and
    \item $T/M$ is countable.
\end{enumerate}

In Theorem \ref{4.1new}, we generalize this result to countable quasi-excellent local domains. In particular, we show that if $T$ is a complete local ring containing the rationals and if $M$ is the maximal ideal of $T$, then $T$ is the completion of a countable quasi-excellent local domain if and only if $T$ is reduced and $T/M$ is countable.  Note that, once again and not surprisingly, the condition that $T$ is equidimensional is not needed. To prove Theorem \ref{4.1new}, we rely heavily on a result proved in \cite{Teresa} (See Remark \ref{important}).

The difference between excellent and quasi-excellent rings is that, by definition, excellent rings are universally catenary, while quasi-excellent rings are not required to be universally catenary. Thus, quasi-excellent rings can be noncatenary. For many years there was doubt that a noncatenary Noetherian ring actually existed. Nagata settled the question in 1956 when he constructed a family of noncatenary local domains by ``gluing together" maximal ideals of different heights in a semilocal domain to produce a noncatenary local (Noetherian) domain [\cite{Nagata}, p. 203]. Heitmann later showed in \cite{HeitmannNonCat} that, given any finite partially ordered set $X$, there exists a Noetherian domain $R$ such that $X$ can be embedded into the prime spectrum of $R$ in a way that preserves saturated chains.  This result shows that there there is no finite bound on how noncatenary a local domain can be. In \cite{Small2017}, Avery et al. further extended Heitmann's and Nagata's work by classifying completions of noncatenary local domains.


Previous work on noncatenary Noetherian domains motivates one to ask how noncatenary a quasi-excellent domain can be.  In fact, in a recent article (see \cite{Colbert}), the authors show that, given any finite partially ordered set $X$, there exists a quasi-excellent local domain $R$ such that $X$ can be embedded into the prime spectrum of $R$ in a way that preserves saturated chains. As a result, there is no finite bound on how noncatenary a quasi-excellent local domain can be. In this paper, we explore the relationship between noncatenary quasi-excellent local domains and their completions. 
In Corollary \ref{noncatcompletions} we characterize completions of noncatenary quasi-excellent local domains.  In fact, in Theorem \ref{specified length chains}, we answer the following more specific question. Let $m_1, m_2, \ldots ,m_n$ be positive integers. If $T$ is a complete local ring of characteristic zero, under what conditions is there a quasi-excellent local domain $A$ with completion $T$ such that $A$ contains maximal saturated chains of prime ideals of lengths $m_1,m_2, \ldots ,m_n$? (Here, by a maximal saturated chain of prime ideals in the local domain $A$, we mean a saturated chain of prime ideals that starts at $(0)$ and ends at the maximal ideal of $A$.) We address this question in Section 3 and Section 4.

In Section 3, we prove the following result (See Theorem \ref{specified length chains}). Let $T$ be a complete local ring of characteristic zero and let $m_1, \ldots ,m_n$ be positive integers with $1 < m_1 < \cdots < m_n < \mbox{dim}(T)$. We show that $T$ is the completion of a quasi-excellent local domain with maximal saturated chains of prime ideals of lengths $m_1, \ldots ,m_n$ if and only if no integer of $T$ is a zerodivisor, $T$ is reduced, and there exist minimal prime ideals $P_1, \ldots ,P_n$ of $T$ such that dim$(T/P_i) = m_i$ for $1 \leq i \leq n$. The proof of this theorem relies heavily on a domain $A$ constructed in \cite{Susan} and a bijection pointed out in that paper between prime ideals of $T$ that are not minimal and nonzero prime ideals of $A$. An application of our results is that we find sufficient conditions on a complete local ring $T$ of characteristic zero so that $T$ is the completion of a quasi-excellent catenary local domain, but is not the completion of a universally catenary local domain.  (See Theorem \ref{cat not uni cat}.)

In Section 4, we prove analogous theorems about the lengths of maximal saturated chains in \textit{countable} quasi-excellent local domains. To classify when $T$ is the completion of a countable quasi-excellent local domain $A$ with maximal saturated chains of specified lengths, we use powerful tools from \cite{Teresa}. Given a complete local ring $T$ satisfying the necessary conditions, we construct a countable subring $R$ of $T$ that contains generators of strategically chosen prime ideals of $T$.  We then find a countable quasi-excellent local subring $S$ of $T$ such that the completion of $S$ is $T$ and such that $S$ contains $R$. We show that since $S$ contains generators of our strategically chosen prime ideals of $T$, it must have maximal saturated chains of prime ideals of the desired lengths.

In Section 5, we show that our results can be used to construct interesting examples of noncatenary quasi-excellent local domains. In particular, in Example \ref{horizontal}, we construct a quasi-excellent local domain $A$ that has infinitely many height one prime ideals $P_i$ for which $A/P_i$ is noncatenary. 

Throughout the paper, when we say a ring is \textit{local}, we mean it has a unique maximal ideal and is Noetherian. We use the term \textit{quasi-local} to mean a ring that has a unique maximal ideal that is not necessarily Noetherian.  If $A$ is a local ring or a quasi-local ring with unique maximal ideal $M$, we denote it $(A,M)$. We define $\widehat{A}$ to be the completion of a local ring $A$ at its maximal ideal, and we say that the length of a chain of prime ideals $P_0 \subsetneq \cdots \subsetneq P_n$ is $n$. A \textit{precompletion} of a complete local ring $T$ is a local ring $A$ whose completion is $T$, i.e. $\widehat{A} \cong T$. 

\section{Preliminaries}


The primary goal of this paper is to generalize previously known results about excellent domains and about noncatenary domains to quasi-excellent domains. We begin by providing some background on quasi-excellent rings.


We call a ring $A$ \textit{catenary} if, for all pairs of prime ideals $P$ and $Q$ of $A$ with $P \subseteq Q$, 
all saturated chains of prime ideals between $P$ and $Q$ have the same length. Otherwise, we call $A$ \textit{noncatenary}. If all finitely generated $A$-algebras are catenary, we call $A$ \textit{universally catenary}. A Noetherian ring $A$ is defined to be \textit{quasi-excellent} if the following two conditions hold:

(i)  For all prime ideals $P$ of $A$, the ring $\widehat{A} \otimes_A L$ is regular for every finite field extension $L$ of $A_P/PA_P$.

(ii) Reg$(B) \subset \Spec(B)$ is open for every finitely generated $A$-algebra $B$.

\noindent A quasi-excellent ring that is universally catenary is said to be an \textit{excellent} ring. Note that a quasi-excellent ring can be noncatenary, catenary, or universally catenary. 

\begin{rmk}\label{quasi-excllent}
Suppose that $A$ is a local ring.  Then, if $A$ satisfies condition (i) of being quasi-excellent, it also satifies condition (ii) of being quasi-excellent (See, for example, Chapter 13, Theorem 76 and Lemma 33.4 in \cite{MatsumuraPink}).   It is also known that a local ring $A$ satisfies condition (ii) of being quasi-excellent if, for all $P \in \Spec(A)$, the ring $\widehat{A} \otimes_A L$ is regular for every purely inseparable finite field extension $L$ of $k(P) = A_P/PA_P$ (See, for example, \cite{Rotthaus}). Thus, if $A$ is a local ring and, for all $P \in \Spec(A)$, the ring $\widehat{A} \otimes_A L$ is regular for every purely inseparable finite field extension $L$ of $k(P)$, then $A$ is quasi-excellent.
\end{rmk}

We now state two important theorems that will be used in Section 3. 

\begin{theorem}[\cite{Matsumura}, Theorem 31.7] \label{universally cat equiv}
     For a Noetherian local ring $A$, the following conditions are equivalent:
    \begin{enumerate}[(i)]
        \item $A$ is universally catenary,
        \item $A[x]$ is catenary, and
        \item $\widehat{A/P}$ is equidimensional for every $P \in \Spec (A)$.
    \end{enumerate}
\end{theorem}

 
\begin{corollary}[\cite{Matsumura}, Corollary 23.9]\label{Matsumura 32.9}
     Let $(A,M)$ and $(B,N)$ be Noetherian local rings and $A \longrightarrow B$ a local homomorphism. Suppose $B$ is flat over $A$.  We have
    \begin{enumerate}[(i)]
        \item If $B$ is normal (or reduced), then so is $A$;
        \item If both $A$ and the fiber rings of $A \longrightarrow B$ are normal (or reduced), then so is $B$.
    \end{enumerate}
\end{corollary}

Recall that, if $A$ is a local ring, then $\widehat{A}$ is a faithfully flat and regular extention of $A$. 

The next result is the main tool for our results in Section 3. Recall that the generic formal fiber ring of a local domain $A$ is defined to be $\widehat{A} \otimes_A QF(A)$ where $QF(A)$ is the quotient field of $A$.  Since there is a one to one correspondence between the prime ideals in the generic formal fiber ring of $A$ and the prime ideals $Q$ of $\widehat{A}$ satisfying $Q \cap A = (0)$, one can informally think of the prime ideals of the generic formal fiber ring of $A$ as being the prime ideals $Q$ of $\widehat{A}$ that satisfy $Q \cap A = (0)$.  Also recall that, for a reduced ring, the set of associated prime ideals is the same as the set of minimal prime ideals.  
So we have that $(iii)$ in the following lemma is equivalent to the statement that $Q \in \Spec(T)$ satisfies $Q \cap A = (0)$ if and only if $Q$ is a minimal prime ideal of $T$.


\begin{lemma}[\cite{Susan}, Lemma 8]\label{Lemma 8}
    Let $(T,M)$ be a complete local reduced ring of dimension at least one. Suppose no integer of $T$ is a zero divisor. Then there exists a local domain $A$ such that
    \begin{enumerate}[(i)]
        \item $\widehat{A} = T$,
        \item If $P$ is a nonzero prime ideal of $A$, then $T \otimes_A k(P) \cong k(P)$ where $k(P) = A_P/PA_P$, \label{l8 - ii}
        \item The generic formal fiber ring of $A$ is semilocal with maximal ideals the associated prime ideals of $T$, and
        \item If $I$ is a nonzero ideal of $A$, then $A/I$ is complete.
    \end{enumerate}
\end{lemma}



\begin{rmk} \label{bijection}
In the last two paragraphs of the proof in \cite{Susan} of the above lemma, it is shown that there is a one to one correspondence between the prime ideals of $T$ that are not associated prime ideals of $T$ and the nonzero prime ideals of $A$.  In fact, it is shown that if $P$ is a nonzero prime ideal of $A$ then $PT$ is a prime ideal of $T$ and it is the only prime ideal of $T$ that lies over $P$. 
\end{rmk}

When we construct quasi-excellent domains that have maximal saturated chains of prime ideals of prescribed lengths, the one to one correspondence given in Remark \ref{bijection} will prove to be very useful. 


The following lemma provides conditions for a local ring $T$ to have a saturated chain of prime ideals of length $n$ with some nice properties. One of our goals in Section 3 and Section 4 is to start with a complete local ring $T$ with saturated chains of prime ideals of some specific lengths and find a quasi-excellent local domain $A$ with maximal saturated chains of prime ideals of the same lengths so that $\widehat{A}=T$. With this in mind, the following lemma becomes very important, as it provides conditions for $T$ to have these saturated chains of prime ideals of specific lengths that satisfy some desirable properties.

\begin{lemma}[\cite{Small2017}, Lemma 2.8]\label{Lemma 2.8}
    Let $(T,M)$ be a local ring with $M$ not in $\Ass(T)$ and let $P$ be in $\Min(T)$ with $\dim(T/P) = n$. Then there exists a saturated chain of prime ideals of $T$, $P \subsetneq Q_1 \subsetneq \cdots \subsetneq Q_{n-1} \subsetneq M$, such that, for each $i=1,\ldots,n-1$, $Q_i$ is not in $\Ass (T)$ and $P$ is the only minimal prime ideal of $T$ contained in $Q_i$.
\end{lemma} 

The following lemma shows that if a local domain has a maximal saturated chain of prime ideals of length $n$, then $\widehat{A}$ must have a minimal prime ideal $P$ such that   dim$(T/P) = n$.

\begin{lemma}[\cite{Small2017}, Lemma 2.9]\label{Lemma 2.9}
    Let $(T,M)$ be a complete local ring and let $A$ be a local domain such that $\widehat{A} \cong T$. If $A$ contains a saturated chain of prime ideals from $(0)$ to $M \cap A$ of length $n$, then there exists $P$ in $\Min(T)$ such that $\dim(T/P) = n$.
\end{lemma} 

The next theorem characterizes completions of noncatenary local domains. We use it to prove results about noncatenary quasi-excellent domains. In addition, we use it in Section 3 to find sufficient conditions for a complete local ring to be the completion of a quasi-excellent catenary local domain but not the completion of a universally catenary local domain.


\begin{theorem}[\cite{Small2017}, Theorem 2.10]\label{Theorem 2.10}
    Let $(T,M)$ be a complete local ring. Then $T$ is the completion of a noncatenary local domain $A$ if and only if the following conditions hold:
    \begin{enumerate}[(i)]
        \item No integer of $T$ is a zero divisor,
        \item $M$ is not in $\Ass(T)$, and
        \item There exists $P$ in $\Min(T)$ such that $1 < \dim (T/P) < \dim (T)$.
    \end{enumerate}
\end{theorem}



We end this section with a result that characterizes completions of countable local domains.  We use it to prove  our main results in Section 4.

\begin{theorem}[\cite{Barrett}, Corollary 2.15] \label{Yu 2.2}
    Let $(T,M)$ be a complete local ring. Then $T$ is the completion of a countable local domain if and only if
    \begin{enumerate}[(i)]
        \item no integer is a zero divisor of $T$, 
        \item unless equal to $(0)$, $M\notin \Ass(T)$, and 
        \item $T/M$ is countable.
    \end{enumerate}
\end{theorem}

\section{Completions of quasi-excellent local domains}

In this section we first characterize completions of quasi-excellent local domains in the characteristic zero case. We begin with a preliminary lemma showing that, with the additional assumption that the complete local ring $T$ has characteristic zero, the domain $A$ given by Lemma \ref{Lemma 8} is quasi-excellent.

\begin{lemma}\label{showingquasi}
Let $(T,M)$ be a reduced complete local ring of characteristic zero and of dimension at least one.  Suppose no integer of $T$ is a zero divisor.  Then the domain $A$ given by Lemma \ref{Lemma 8} is quasi-excellent.
\end{lemma}

\begin{proof}
    The domain $A$ given by Lemma \ref{Lemma 8} satisfies the properties that $\widehat{A} \cong T$, for every non-zero prime ideal $P$ of $A$, $T \otimes_A k(P) \cong k(P)$ where $k(P) = A_P/PA_P$, and, if $Q$ is a prime ideal of $T$ then $Q \cap A = (0)$ if and only if $Q$ is a minimal prime ideal of $T$.  By Remark \ref{quasi-excllent}, in order to show that $A$ is quasi-excellent, it suffices to show that for all $P \in \Spec(A), T \otimes_A L$ is regular for every purely inseparable finite field extension $L$ of $k(P)$. We treat this in two cases.

If $(0) \neq P \in \Spec(A)$, and $L$ is a finite field extension of $k(P)$, then we have $$T \otimes_A L \cong T \otimes_A k(P) \otimes_{k(P)} L \cong k(P) \otimes_{k(P)} L \cong L,$$ which is a field. Therefore, $T \otimes_A L$ is regular. Now, let $P = (0)$, and let $L$ be a purely inseparable finite field extension of $k(P)$. Since $T$ has characteristic zero, $k((0))$ must also have characteristic zero and so $L = k(P) = k((0))$.  Now $T \otimes_A L = T \otimes_A k((0)) \cong S^{-1} T$, where $S=A \setminus  \{0\}$. Recall that $Q \in \Spec(T)$ satisfies $Q \cap A = (0)$ if and only if $Q$ is a minimal prime ideal of $T$. 
 Therefore, the prime ideals of $S^{-1}T$ are of the form $Q^e$ where $Q$ is a minimal prime ideal of $T$ and $Q^e$ is the image of $Q$ under the natural map $T \longrightarrow S^{-1}T$.  Now $(S^{-1}T)_{Q^e} \cong T_Q$, and since $T$ is reduced and $Q$ is a minimal prime ideal of $T$, we have that $T_Q$ is a field and hence a regular local ring. It follows that $T \otimes_A L$ is a regular ring as desired.
\end{proof}

We are now ready to characterize completions of quasi-excellent domains in the characteristic zero case.

\begin{theorem}\label{quasi-excellent characterization}
 Let $(T,M)$ be a complete local ring of characteristic zero. Then $T$ is the completion of a quasi-excellent local domain if and only if
 \begin{enumerate}[(i)]
     \item No integer of $T$ is a zero divisor, and \label{qec - i}
     \item $T$ is reduced. \label{qec - ii}
 \end{enumerate}
\end{theorem}
\begin{proof}
Suppose $A$ is a quasi-excellent local domain with completion $T$. By Theorem \ref{Lech's Theorem}, since $A$ is a local domain, its completion $T$ has no integer zero divisors.

To show $T$ is reduced, first note that $T$ is a faithfully flat extension of $A$. Thus, by Corollary \ref{Matsumura 32.9}, it suffices to show that both $A$ and the fiber rings of $A \longrightarrow T$ are reduced. Since $A$ is a domain, it is reduced. We now show that the fiber ring $T \otimes_A k(P)$ is reduced for all \mbox{$P \in \Spec(A)$} where $k(P)$ is defined as $A_P/PA_P$. Since $A$ is a quasi-excellent, $(T \otimes_A k(P))_Q$ is a regular local ring for every \mbox{$Q \in \Spec (T \otimes_A k(P))$}. This implies that $(T \otimes_A k(P))_Q$ has no non-zero nilpotent elements for all $P \in \Spec A$ and for all $Q \in \Spec(T \otimes_A k(P))$. It follows that the fiber rings of $A \rightarrow T$ have no non-zero nilpotent elements and so $T$ is reduced.

Conversely, suppose $(T,M)$ is a reduced complete local ring of characteristic zero such that no integer is a zero divisor. If $\dim(T)=0$, then $T$ is a field, and is therefore its own completion. As complete local rings are excellent, the result follows in this case.  Now suppose $\dim(T) \geq 1$ and let $A$ be the domain given by Lemma \ref{Lemma 8}. By Lemma \ref{showingquasi}, $A$ is quasi-excellent.
\end{proof}

Now that we have characterized the completions of quasi-excellent local domains, we ask more specific questions. For example, when does a complete local ring have a quasi-excellent local domain precompletion that is not catenary? Theorem \ref{specified length chains} provides insight into an answer to this question by showing that certain complete local rings have quasi-excellent precompletions that have maximal saturated chains of prime ideals of different lengths and hence, are not catenary. 


\begin{theorem}\label{specified length chains}
 Let $(T,M)$ be a complete local ring of characteristic zero and let $m_i$ for $i = 1,2, \ldots ,n$ be integers such that $1 < m_1 < \dots < m_n < \dim (T)$. Then $T$ is the completion of a quasi-excellent local domain with maximal saturated chains of prime ideals of lengths $m_1,\ldots,m_n$ if and only if
 \begin{enumerate}[(i)]
     \item No integer of $T$ is a zero divisor, \label{slc - i}
     \item $T$ is reduced, and \label{slc - ii}
     \item There exists prime ideals $P_1,\ldots,P_n$ in $\Min (T)$ such that $\dim (T/P_i) = m_i$ for $1 \leq i \leq n$.\label{slc - iii}
 \end{enumerate}
\end{theorem}
\begin{proof}
        Suppose $A$ is a quasi-excellent domain such that $\widehat{A} \cong T$ and that there exist maximal saturated chains of prime ideals of lengths $m_1,\dots,m_n$ in $A$. It follows from Theorem \ref{quasi-excellent characterization} that conditions \eqref{slc - i} and \eqref{slc - ii} are satisfied. By Lemma \ref{Lemma 2.9} there exists $P_i \in \Min(T)$ such that $\dim(T/P_i)=m_i$ for each $i \in \{1,\ldots,n\}$. Therefore, condition \eqref{slc - iii} is also satisfied.

        Now suppose $T$ is a complete local ring of characteristic zero satisfying conditions \eqref{slc - i}-\eqref{slc - iii}. By condition (iii), there exist saturated chains of prime ideals of $T$
\begin{align*}
            P_i\subsetneq Q_{i,1}\subsetneq \dots \subsetneq Q_{i, m_i-1} \subsetneq M
        \end{align*}
for each $i \in \{1,2, \ldots ,n\}$.
        Since dim$(T/P_1) > 1$, we have dim$(T) > 1$. Let $A$ be the local domain given by Lemma \ref{Lemma 8}, and note that by Lemma \ref{showingquasi}, $A$ is quasi-excellent. Intersecting these chains with $A$ using Remark \ref{bijection} and the fact that $P_i \cap A = (0)$ for all $i \in \{1,2, \ldots ,n\}$, we obtain $n$ chains of prime ideals of $A$
$$(0)\subsetneq Q_{i,1} \cap A\subsetneq \dots \subsetneq Q_{i, m_i-1} \cap A \subsetneq M\cap A$$
        for all $i \in \{1,2, \ldots ,n\}$. By Remark \ref{bijection}, each of these chains is saturated, and it follows that $A$ has maximal saturated chains of prime ideals of lengths $m_1, \ldots ,m_n$.
        %
        %
       %
        %
\end{proof}

If $T$ satisfies the conditions in Theorem \ref{specified length chains} with at least one $m_i$ such that $1 < m_i < n = \dim (T)$, then the domain $A$ obtained from Theorem \ref{specified length chains} will be a quasi-excellent local domain with a maximal saturated chain of prime ideals of length $m_i$, and, since dim$(A) = \mbox{dim}(T)$, a maximal saturated chain of prime ideals of length $n$. It follows that $A$ is not catenary.  In this way, Theorem \ref{specified length chains} can be used to construct quasi-excellent local domains that are not catenary with some control on the lengths of the saturated maximal chains of prime ideals.  We note that it is not particularly difficult to find complete local rings that satisfy the conditions of Theorem \ref{specified length chains}.  For example, $T=\mathbb{Q}[[x,y,z,w,t]]/((x) \cap (y,z))$ has a precompletion that is a quasi-excellent local domain with maximal saturated chains of prime ideals of length three and of length four.



We now state, as a corollary, the special case of Theorem \ref{specified length chains} for $n = 1$.

\begin{corollary} \label{corrolary one m}
Let $(T,M)$ be a complete local ring of characteristic zero and let $m$ be an integer with $1 < m\leq \mbox{dim}(T)$. Then $T$ is the completion of a quasi-excellent local domain with a maximal saturated chain of prime ideals of length $m$ if and only if
 \begin{enumerate}[(i)]
     \item No integer of $T$ is a zero divisor, \label{com - i}
     \item $T$ is reduced, and \label{com - ii}
     \item There exists a prime ideal $P$ in $\Min (T)$ such that $\dim (T/P) = m$. \label{com - iii}
 \end{enumerate}
\end{corollary}

\begin{proof}
The result follows from Theorem \ref{specified length chains} if $m < \mbox{dim}(T)$. So assume $m = \mbox{dim}(T)$. If $T$ is the completion of a quasi-excellent local domain with a maximal saturated chain of prime ideals of length $m$, then by Theorem \ref{quasi-excellent characterization}, the first two conditions are satisfied.  Since dim$(T) = m$, there must be a minimal prime ideal $P$ of $T$ such that dim$(T/P) = m$.  If the three conditions hold for $T$, then, by Theorem \ref{quasi-excellent characterization}, $T$ is the completion of a quasi-excellent local domain $A$.  Since dim$(A) = \mbox{dim}(T)$, $A$ must have a maximal saturated chain of prime ideals of length $m$.
\end{proof}

Our results allow us to classify completions of noncatenary quasi-excellent domains in the characteristic zero case.

\begin{corollary} \label{noncatcompletions}
Let $(T,M)$ be a complete local ring of characteristic zero. Then $T$ is the completion of a noncatenary quasi-excellent local domain if and only if
 \begin{enumerate}[(i)]
     \item No integer of $T$ is a zero divisor, 
     \item $T$ is reduced, and 
     \item There exists a prime ideal $P$ in $\Min(T)$ such that $1 < \dim (T/P) < \dim (T)$.
 \end{enumerate}
\end{corollary}

\begin{proof}
 Suppose $T$ is the completion of a noncatenary quasi-excellent local domain. By Theorem \ref{quasi-excellent characterization}, no integer of $T$ is a zero divisor and $T$ is reduced. By Theorem \ref{Theorem 2.10} condition $(iii)$ holds.  Now suppose $T$ satisfies conditions $(i)$, $(ii)$, and  $(iii)$. By Corollary \ref{corrolary one m}, $T$ is the completion of a quasi-excellent local domain $A$ with a maximal saturated chain of prime ideals of length dim$(T/P) < \mbox{dim}(T) = \mbox{dim}(A)$. It follows that $A$ is not catenary.
\end{proof}

We now state another interesting consequence of Theorem \ref{quasi-excellent characterization}.  In particular, we give sufficient conditions for a complete local ring $T$ to be the completion of a quasi-excellent catenary local domain, but not the completion of a universally catenary local domain.  In other words, although $T$ is the completion of a domain that is ``almost'' excellent, it is not the completion of an excellent domain.

\begin{theorem}\label{cat not uni cat}
Let $(T,M)$ be a complete local ring of characteristic zero and dimension $n > 1$. Then $T$ is the completion of a quasi-excellent catenary local domain but is not the completion of a universally catenary local domain if the following conditions hold
 \begin{enumerate}[(i)]
     \item No integer of $T$ is a zero divisor, \label{cnu - i}
     \item $T$ is reduced, \label{cnu - ii}
     \item For every $P \in \Min (T)$, either $ \dim(T/P) = 1$ or $ \dim(T/P) = n $, and \label{cnu - iii}
     \item There exist $P_0, P_1 \in \Min (T)$ such that $\dim(T/P_0) = 1$ and $\dim(T/P_1) = n$. \label{cnu - iv}
 \end{enumerate}
\end{theorem}
\begin{proof}
    Suppose $T$ is a complete local ring of characteristic zero satisfying conditions \eqref{cnu - i}-\eqref{cnu - iv}. By Theorem \ref{quasi-excellent characterization}, there exists a quasi-excellent local domain $A$ such that $\hat{A} \cong T$. By Theorem \ref{Theorem 2.10}, $A$ is catenary. 
    Now suppose $T$ is the completion of a universally catenary local domain $S$. 
 Then, by Theorem \ref{universally cat equiv}, $\widehat{S} = T$ is equidimensional, contradicting that $T$ satisfies condition (iv).
\end{proof}

\section{Completions of countable quasi-excellent local domains}

In this section we prove analogous versions of the main results from Section 3 for countable quasi-excellent local domains. 
  We start with the analogous version of Theorem \ref{quasi-excellent characterization}.  That is, in Theorem \ref{4.1new}, we characterize completions of countable quasi-excellent local domains in the case where the complete local ring contains the rationals.   We note that, although the statement of Theorem \ref{4.1new} is not currently in the literature, the proof of the theorem is essentially found in the proof of Theorem 3.9 in \cite{Teresa}. For convenience, we state that theorem here.

   \begin{theorem}[\cite{Teresa}, Theorem 3.9]
 \label{Yu 3.9new}
     Let $(T,M)$ be a complete local ring containing the rationals. Then $T$ is the completion of a countable excellent domain if and only if the following conditions hold:
     \begin{enumerate}
         \item $T$ is equidimensional,
         \item $T$ is reduced, and
         \item $T/M$ is countable.
    \end{enumerate}
 \end{theorem}

 \begin{rmk}\label{important}
In the proof of the above theorem in \cite{Teresa}, the authors show the following result.  Let $(T,M)$ be a reduced complete local ring containing the rationals such that dim$(T) \geq 1$ and $T/M$ is countable.  Let $(R_0,R_0 \cap M)$ be a countable local subring of $T$ such that $\widehat{R_0} = T$.  Then there exists a subring $(S, S \cap M)$ of $T$ such that $S$ is a countable quasi-excellent local domain, $R_0 \subseteq S$, and $\widehat{S} = T$.
 \end{rmk}

 In Theorem \ref{4.1new}, we show that the  quasi-excellent version of Theorem \ref{Yu 3.9new} is the same as Theorem \ref{Yu 3.9new} except that the ``$T$ is equidimensional'' condition is omitted.

\begin{theorem} \label{4.1new}
    Let $(T,M)$ be a complete local ring containing the rationals. Then $T$ is the completion of a countable quasi-excellent local domain if and only if
    \begin{enumerate}[(i)]
        \item $T$ is reduced, and
        \item $T/M$ is countable.
    \end{enumerate}
\end{theorem}

\begin{proof}
If $T$ is the completion of a countable quasi-excellent domain, then, by Theorem \ref{quasi-excellent characterization}, $T$ is reduced and by Theorem \ref{Yu 2.2}, $T/M$ is countable.

Now suppose that $T$ is reduced and $T/M$ is countable.  If dim$(T) = 0$ then, since $T$ is reduced, $T$ is a field and $M = (0)$.  It follows that $T/M \cong T$ is countable, so $T$ is the completion of a countable quasi-excellent local domain, namely itself.  If dim$(T) \geq 1$, then, since $T$ contains the rationals, no integer of $T$ is a zero divisor.  Because $T$ is reduced and dim$(T) \geq 1$ we have that $M$ is not an associated prime ideal of $T$.  By Theorem \ref{Yu 2.2}, $T$ is the completion of a countable local domain $(R_0, R_0 \cap M)$. By Remark \ref{important}, we have that $T$ is the completion of a countable quasi-excellent local domain.
%
\end{proof}

For the remainder of this section, we focus on an analogous version of Theorem \ref{specified length chains} where we require $A$ to be countable.  We begin by recalling a definition and two important results from \cite{Teresa}.

\begin{definition}[\cite{Teresa}, Definition 3.2] Let $(T,M)$ be a complete local ring and let $(R_0, R_0 \cap M)$ be a countable local subring of $T$ such that $R_0$ is a domain and $\widehat{R_0}=T$. Let $(R, R \cap M)$ be a quasi-local subring of $T$ with $R_0 \subseteq R$. Suppose that
\begin{enumerate}[(i)]
    \item $R$ is countable, and
    \item $R \cap P = (0)$ for every $P \in \Ass(T)$.
\end{enumerate}
Then we call $R$ a built-from-$R_0$ subring of $T$, or a $BR_0$-subring of $T$ for short.
\end{definition}

\begin{lemma}[\cite{Teresa}, Lemma 3.6] \label{Yu 3.6new}
    Suppose $(T,M)$ is a complete local ring with $\dim (T)\geq 1$, $(R_0, R_0\cap M)$ is a countable local domain with $R_0\subseteq T$ and $\widehat{R_0} = T$, and $(R, R\cap M)$ is a $BR_0$-subring of $T$. Let $Q\in \Spec(T)$ such that $Q\nsubseteq P$ for all $P\in \Ass(T)$. Then there exists a $BR_0$-subring of $T$, $(R', R'\cap M)$ such that $R\subseteq R'$ and $R'$ contains a generating set for $Q$. 
\end{lemma}

\begin{lemma}[\cite{Teresa}, Lemma 3.8] \label{Yu 3.8new}
    Suppose $(T,M)$ is a complete local ring with $\dim (T)\geq 1$ and $(R_0, R_0\cap M)$ is a countable local domain with $R_0\subseteq T$ and $\widehat{R_0} = T$. Let $(R, R\cap M)$ be a $BR_0$-subring of $T$. Then there exists a $BR_0$-subring of $T$, $(R', R'\cap M)$, such that $R\subseteq R' \subseteq T$, and, if $I$ is a finitely generated ideal of $R'$, then $IT \cap R' = IR'$. Thus, $R'$ is Noetherian and $\widehat{R'} = T$.
\end{lemma}

Lemma \ref{Yu 3.6new} shows that, given a $BR_0$-subring of $T$, one can find a larger $BR_0$-subring of $T$ that contains a generating set for one particular prime ideal $Q$ of $T$. To prove the analogous version of Theorem \ref{specified length chains} for countable quasi-excellent domains, we show that the same can be done for a countable collection $Q_1, Q_2, \ldots$ of prime ideals of $T$. In fact, for our proof of Theorem \ref{specified length chains countable}, we only need to do this for a finite collection of prime ideals of $T$, but it is not difficult to prove it for countably infinitely many, and so we state and prove the result for a countable collection of prime ideals of $T$. We prove the result using repeated applications of Lemma \ref{Yu 3.6new}.

\begin{lemma}\label{countableQs}
    Suppose $(T,M)$ is a complete local ring with $\dim (T)\geq 1$, $(R_0, R_0\cap M)$ is a countable local domain with $R_0\subseteq T$ and $\widehat{R_0} = T$, and $(R, R\cap M)$ is a $BR_0$-subring of $T$. Let $Q_1, Q_2, \ldots$ be a countable set of prime ideals of $T$ such that, for every $i = 1,2, \ldots$, $Q_i\nsubseteq P$ for all $P\in \Ass(T)$. Then there exists a $BR_0$-subring of $T$, $(R', R'\cap M)$ such that $R\subseteq R'$ and, for every $i = 1,2, \ldots, $ $R'$ contains a generating set for $Q_i$. 
\end{lemma}

\begin{proof}
We inductively define an ascending chain $R \subseteq R_1 \subseteq R_2 \subseteq \cdots$ of $BR_0$-subrings of $T$ such that, for all $i = 1,2, \ldots$, $R_i$ contains a generating set for $Q_i$. Let $(R_1,R_1, \cap M)$ be the $BR_0$-subring obtained from Lemma \ref{Yu 3.6new} so that $R \subseteq R_1$ and $R_1$ contains a generating set for $Q_1$.  Let $(R_2,R_2, \cap M)$ be the $BR_0$-subring obtained from Lemma \ref{Yu 3.6new} so that $R_1 \subseteq R_2$ and $R_2$ contains a generating set for $Q_2$.  Continue the process to define $R_i$ for every $i = 1,2, \dots$.  
Define $R' = \bigcup_{i = 1}^{\infty}R_i$.  Then $R'$ is countable and, for every $i$, if $P \in \Ass(T)$, then $R_i \cap P = (0)$.  It follows that $R' \cap P = (0)$ for every $P \in \Ass(T)$.  Hence, $R'$ is a $BR_0$-subring of $T$.  By construction, $R'$ contains a generating set for every $Q_i$.
\end{proof}

We now have all of the tools needed to prove the analogous version of Theorem \ref{specified length chains} for countable quasi-excellent domains.

\begin{theorem}\label{specified length chains countable}
 Let $(T,M)$ be a complete local ring containing the rationals and let $m_i$ for $i = 1,2, \ldots ,n$ be integers such that $1 < m_1 < \dots < m_n < \dim (T)$. Then $T$ is the completion of a countable quasi-excellent local domain with maximal saturated chains of prime ideals of lengths $m_1,\ldots,m_n$ if and only if
 \begin{enumerate}[(i)]
     \item $T$ is reduced, 
     \item There exists prime ideals $P_1,\ldots,P_n$ in $\Min (T)$ such that $\dim (T/P_i) = m_i$ for $1 \leq i \leq n$, and 
     \item $T/M$ is countable.
 \end{enumerate}
\end{theorem}

\begin{proof}
Suppose $T$ is the completion of a countable quasi-excellent domain with maximal saturated chains of prime ideals of lengths $m_1,\ldots,m_n$. By Theorem \ref{specified length chains}, $T$ satisfies the first two conditions of the theorem and by Theorem \ref{Yu 2.2}, $T/M$ is countable.

Now suppose $T$ satisfies conditions $(i), (ii)$ and $(iii)$.  Note that, by hypothesis, dim$(T) \geq 1$. Since $T$ is reduced and dim$(T) \geq 1$, $M$ is not an associated prime ideal of $T$. Now use Lemma \ref{Lemma 2.8} to find, for every $i = 1,2, \ldots ,n$, a saturated chain of prime ideals of $T$
$$P_i \subsetneq Q_{i,1} \subsetneq \cdots \subsetneq Q_{i, m_i - 1} \subsetneq M$$
where, for $j = 1,2, \ldots, m_i - 1$, $Q_{i,j}$ is not in $\Ass(T)$ and $P_i$ is the only minimal prime ideal contained in $Q_{i,j}$.

Note that, since $T$ contains the rationals, no integer of $T$ is a zero divisor.  By Theorem \ref{Yu 2.2}, $T$ is the completion of a countable local domain $(R_0,R_0 \cap M)$.  Since $T$ is reduced, the set of associated prime ideals of $T$ is the same as the set of minimal prime ideals of $T$.  Hence, we have that $Q_{i,j} \not\subseteq P$ for all $P \in \Ass(T)$. Now use Lemma \ref{countableQs} to find a $BR_0$-subring of $T$, $(R',R' \cap M)$ such that $R_0 \subseteq R'$ and $R'$ contains a generating set for all $Q_{i,j}$ where $i = 1,2, \ldots,n$ and $j = 1,2, \ldots, m_i$. Now use Lemma \ref{Yu 3.8new} to obtain a $BR_0$-subring of $T$, $(R'',R''\cap M)$ such that $R' \subseteq R''$, $R''$ is Noetherian, and $\widehat{R''} = T$.  Finally, use Remark \ref{important} to find a subring $(S,S \cap M)$ of $T$ such that $S$ is a countable quasi-excellent local domain, $R'' \subseteq S$, and $\widehat{S} = T$.  We claim that $S$ has maximal saturated chains of prime ideals of lengths $m_1, \ldots ,m_n$.

Note that, since $R' \subseteq S$, $S$ contains a generating set for all $Q_{i,j}$ where $i = 1,2, \ldots,n$ and $j = 1,2, \ldots ,m_i$.  Fix $i$ and intersect the chain $P_i \subsetneq Q_{i,1} \subsetneq \cdots \subsetneq Q_{i, m_i - 1} \subsetneq M$ with $S$ to obtain the chain of prime ideals
$$(0) = P_i \cap S\subsetneq Q_{i,1} \cap S \subsetneq \cdots \subsetneq Q_{i, m_i - 1} \cap S \subsetneq M \cap S$$
of $S$.  Note that the containments in this chain are strict since $S$ contains a generating set for each $Q_{i,j}$.  We claim this chain is saturated.  To see this, first observe that $(Q_{i,m_i - 1} \cap S)T = Q_{i,m_i - 1}$. 
 Hence, the completion of $S/(Q_{i,m_i - 1} \cap S)$ is $T/Q_{i,m_i - 1}$. As dim$(T/Q_{i,m_i - 1}) = 1$, we have that dim$(S/(Q_{i,m_i - 1} \cap S)) = 1$ and so the chain $Q_{i, m_i - 1} \cap S \subsetneq M \cap S$ is saturated.  Suppose the chain
 $$(0) = P_i \cap S\subsetneq Q_{i,1} \cap S \subsetneq \cdots \subsetneq Q_{i, m_i - 1} \cap S$$
 is not saturated.  Then, by the Going Down Property, there is a chain of prime ideals of $T$, $J_0 \subsetneq J_1 \subsetneq \cdots \subsetneq Q_{i,m_i - 1}$ of length greater than $m_i - 1$ where $J_0$ is a minimal prime ideal of $T$. Since $P_i$ is the only minimal prime ideal contained in $Q_{i,m_i - 1}$, we have that $J_0 = P_i$.  Therefore, there exist two saturated chains of prime ideals from $P_i$ to $Q_{i,m_i - 1}$ of different lengths.  Since $T$ is a complete local ring, it is excellent, and hence catenary, and so this is impossible.  It follows that the chain of prime ideals
 $$(0) = P_i \cap S\subsetneq Q_{i,1} \cap S \subsetneq \cdots \subsetneq Q_{i, m_i - 1} \cap S \subsetneq M \cap S$$
 of $S$ is saturated.  Hence, for every $i$, $S$ contains a maximal saturated chain of prime ideals of length $m_i$.
\end{proof}

We end this section by classifying completions of countable noncatenary quasi-excellent domains in the case where the complete local ring contains the rationals.

\begin{corollary} \label{countablenoncatcompletions}
Let $(T,M)$ be a complete local ring containing the rationals. Then $T$ is the completion of a countable noncatenary quasi-excellent local domain if and only if
 \begin{enumerate}[(i)]
     \item $T$ is reduced, 
     \item There exists a prime ideal $P$ in $\Min(T)$ such that $1 < \dim (T/P) < \dim (T)$, and
     \item $T/M$ is countable.
 \end{enumerate}
\end{corollary}

\begin{proof}
Suppose that $T$ is the completion of a countable noncatenary quasi-excellent local domain.  By Corollary \ref{noncatcompletions}, $T$ satisfies conditions $(i)$, and $(ii)$. By Theorem \ref{Yu 2.2}, $T/M$ is countable.  Conversely, suppose $T$ satifies conditions $(i)$, $(ii)$, and $(iii)$. By Theorem \ref{specified length chains countable}, $T$ is the completion of a countable quasi-excellent local domain $A$ with a maximal saturated chain of prime ideals of length dim$(T/P) < \mbox{dim}(T) = \mbox{dim}(A)$. As a consequence, $A$ is not catenary.
\end{proof}


\section{Examples of noncatenary quasi-excellent domains}


The results in this paper can be used to construct interesting examples of noncatenary quasi-excellent domains.  For example, suppose $n$ and $m$ are integers with $1 < n < m$, and let $T = \mathbb{Q}[[x_1,\ldots,x_{m + 1}]]/(P \cap Q)$ where $P$ and $Q$ are primes ideals of $\mathbb{Q}[[x_1,\ldots,x_{m + 1}]]$ with $P \not\subseteq Q$, $Q \not\subseteq P$, dim$(T/P) = n$ and dim$(T/Q) = m$. Then by Theorem \ref{specified length chains}, $T$ is the completion of a quasi-excellent local domain $A$ with maximal saturated chains of prime ideals of lengths $n$ and $m$. In fact, by Theorem \ref{specified length chains countable}, a countable such $A$ exists.  This shows that there is no bound on how noncatenary quasi-excellent domains, or even countable quasi-excellent domains, can be.

It is natural to then ask if quasi-excellent domains can be noncatenary in infinitely many places.  More precisely, we ask whether or not there exists a quasi-excellent domain $A$ such that there are infinitely many prime ideals $\{P_{\alpha} \}_{\alpha \in \Omega}$  of $A$ with $A/P_{\alpha}$ noncatenary for all $\alpha \in \Omega$ where $\Omega$ is an infinite index set.
%
In this section, we provide an example of a quasi-excellent local domain $A$ with uncountably many height one prime ideals $P \in \Spec A$ such that $A/P$ is noncatenary.

\begin{example}\label{horizontal}
Let 
\begin{align*}
    T' = \frac{\Q[[x,y,z,w]]}{(x)\cap (y,z)} 
\end{align*}
Let $A'$ be a quasi-excellent local domain obtained from Theorem \ref{specified length chains} so that $\widehat{A'} = T'$ and $A'$ has maximal saturated chains of prime ideals of length three and length two. Let $M = (x,y,z,w)$ be the maximal ideal of $T'$. 
Define $A := A'[t]_{(M\cap A',t)}$ where $t$ is an indeterminate, and note that $A$ is a quasi-excellent local domain. We claim that $A/P$ is noncatenary for uncountably many height one prime ideals $P$ of $A$. 



First note that if $a \in A' \cap M$, then $A/(t-a) \cong A'$ and so  $(t-a)$ is a height one prime ideal of $A$. Moreover, since $A'$ is noncatenary, $A/(t-a)$ is noncatenary. Now if in the ring $A'[t]$, the ideal $(t - a)$ is equal to the ideal $(t-b)$ where $a,b \in A' \cap M$ and $a \neq b$, then $a-b \in (t-a)$, contradicting that all nonzero elements in $(t - a)$ have degree at least one with respect to the indeterminate $t$.  It follows that in the ring $A$, the ideals $(t - a)$ and $ (t - b)$ are equal if and only if $a = b$.  Now recall that the ring $A'$ constructed using Theorem \ref{specified length chains} satisfies the condition that there is a one to one correspondence between the prime ideals of $T'$ that are not associated prime ideals of $T'$ and the nonzero prime ideals of $A'$ (See Remark \ref{bijection}).  Since $T'$ has uncountably many prime ideals, so does $A'$ and it follows that $A' \cap M$ has uncountably many elements.  Hence, $A$ has uncountably many height one prime ideals $(t - a)$ such that $A/(t - a)$ is noncatenary.
\end{example}

\section{Acknowledgements}
We thank Williams College and the National Science Foundation, via NSF Grant DMS2241623, and NSF Grant DMS1947438 for their generous funding of our research.

\newpage
\bibliography{Bibliography}
\bibliographystyle{plain}

\end{document}